\documentclass{article}

\usepackage{amsthm}
\usepackage{amsfonts}
\usepackage{tikz-cd}
\usepackage{cite}
\usepackage{url}
\usepackage{geometry}

\begin{document}

\title{A Short Proof of a Grothendieck-Lefschetz Theorem for Equivariant Picard Groups}
\author{David Villalobos-Paz}
\date{}
\maketitle

\begin{abstract}
We give a short proof of a Grothendieck-Lefschetz Theorem for equivariant Picard groups of nonsingular varieties with the action of an affine algebraic group.
\end{abstract}

\section{Introduction}
The Grothendieck-Lefschetz Theorem for Picard groups, see Cor XII.3.6 and Cor XII.3.7 in \cite{SGA2} or Cor IV.3.3 in \cite{Hartshorne1970}, states that if $X$ is a nonsingular projective variety over a field of characteristic zero and $Y$ is a subvariety of $X$, such that $\mathrm{dim}(Y) \geq 3$ and $Y$ is a complete intersection in $X$, then the natural map $\mathrm{Pic}(X) \to \mathrm{Pic}(Y)$ is an isomorphism.

If $G$ is an algebraic group acting on $X$, we let $\mathrm{Pic}^G(X)$ denote the equivariant Picard group of $X$, as defined in Chapter 1.3 of \cite{MumfordFogartyKirwan1994}; in other words, the elements of $\mathrm{Pic}^G(X)$ are line bundles over $X$ with a $G$-linearisation, and this is a group via the tensor product. It has been shown recently in \cite{Ravi2017} that a version of Grothendieck-Lefschetz holds for the equivariant Picard group, provided that the group $G$ is assumed to be finite. This was done by adapting the proof of Grothendieck-Lefschetz in \cite{Hartshorne1970} to deal with the group action.

In this paper, we drop the finiteness hypothesis and prove the result for any affine algebraic group. More specifically, we prove the following:

\newtheorem{theorem}{Theorem}

\begin{theorem}

Let $k$ be a field. Let $Y \subset X$ be normal proper $k$-varieties. Let $G$ be an affine algebraic group over $k$ acting on $X$, and suppose $Y$ is G-invariant. If the natural map $\mathrm{Pic}(X) \to \mathrm{Pic}(Y)$ is injective (resp. an isomorphism), then the natural map on equivariant Picard groups $\mathrm{Pic}^G(X) \to \mathrm{Pic}^G(Y)$ is also injective (resp. an isomorphism).

\end{theorem}

As a corollary, we obtain the equivariant version of Grothendieck-Lefschetz:

\newtheorem{corollary}[theorem]{Corollary}

\begin{corollary}

Let $k$ be a field of characteristic zero. Let $X$ be a nonsingular projective variety over $k$ and let $G$ be an affine algebraic group over $k$ acting on $X$. Let $Y$ be a $G$-invariant normal subvariety of $X$, such that $\mathrm{dim}(Y) \geq 3$ and $Y$ is a complete intersection in $X$. Then the natural map $\mathrm{Pic}^G(X) \to \mathrm{Pic}^G(Y)$ is an isomorphism. If instead we assume that $\mathrm{dim}(Y) = 2$, then the natural map $\mathrm{Pic}^G(X) \to \mathrm{Pic}^G(Y)$ is injective.

\end{corollary}

\section*{Acknowledgements}
The author would like to thank his advisor, J\'anos Koll\'ar, for his constant support, encouragement, and very useful observations. He also thanks Ziquan Zhuang for a helpful conversation.

\section{Equivariant Grothendieck-Lefschetz}
Before proving Theorem 1, we need a technical lemma that will give us an exact sequence involving $\mathrm{Pic}^G(X)$. Throughout, we assume that varieties are irreducible.

\newtheorem{lemma}[theorem]{Lemma}

\begin{lemma}
Let $X$ be a normal proper variety over $k$ and $G$ be an affine algebraic group over $k$ acting on $X$. Then there exists an exact sequence of abelian groups

\[ 0 \rightarrow \mathrm{Hom}_\mathrm{GpSch}(G,k^\times) \rightarrow \mathrm{Pic}^G (X) \rightarrow \mathrm{Pic}(X)^G \rightarrow \mathrm{Ext}_\mathrm{GpSch}^1 (G,k^\times).\]

\end{lemma}

\begin{proof}
Let $\mathrm{Ext}_\mathrm{Gp}^1(G,k^\times)$ denote the set of (not necessarily central) extensions of abstract groups, where $k^\times = \mathcal{O}_X(X)^\times$ is viewed as a G-module via the morphism $G \to \mathrm{Aut}(X)$. We first remark that, since $k^\times$ is an abelian group, then the set $\mathrm{Ext}_\mathrm{Gp}^1(G,k^\times)$ is itself endowed with the structure of an abelian group, with the Baer sum as the group operation; see Ex. 7 on page 114 of \cite{MacLane95}. We can then consider the subset $\mathrm{Ext}_\mathrm{GpSch}^1(G,k^\times)$ of extensions of group schemes. In fact, from the definition of the Baer sum, we can see that $\mathrm{Ext}_\mathrm{GpSch}^1(G,k^\times)$ is closed under the group operation, so that it is a subgroup of $\mathrm{Ext}_\mathrm{Gp}^1(G,k^\times).$

We define the map $\epsilon \colon \mathrm{Pic}(X)^G \rightarrow \mathrm{Ext}_\mathrm{GpSch}^1 (G, k^\times)$ as in Section 3.4 of \cite{Brion}. We introduce some notation, so that we can spell out this definition: For a line bundle $L$ over $X$, we define $\mathrm{Aut}(X)_L = \{ \phi \in \mathrm{Aut}(X) \ \colon \ \phi^\star L \cong L\}$ to be the subgroup of $\mathrm{Aut}(X)$ that stabilises $L$, and we let $\mathrm{Aut}^{\mathbb{G}_m} (L)$ denote the group of automorphisms of $L$ that commute with the action of $\mathbb{G}_m$ by multiplication on the fibres of $L$.

By Lemma 3.4.1 in \cite{Brion}, we have an extension $1 \rightarrow k^\times \rightarrow \mathrm{Aut}^{\mathbb{G}_m} (L) \rightarrow \mathrm{Aut}(X)_L \rightarrow 1$. If, furthermore, $L \in \mathrm{Pic}(X)^G$, then the image of the homomorphism $G \rightarrow \mathrm{Aut}(X)$ is contained in $\mathrm{Aut}(X)_L$. Thus, we can pull back the given group extension by the homomorphism $G \rightarrow \mathrm{Aut}(X)_L$, and we obtain another short exact sequence $1 \rightarrow k^\times \rightarrow \mathcal{G}(L) \rightarrow G \rightarrow 1$, where $\mathcal{G}(L)$ is called the lifting group associated to $L$. We then define $\epsilon$ by $L \mapsto (1 \rightarrow k^\times \rightarrow \mathcal{G}(L) \rightarrow G \rightarrow 1).$ We observe that since the morphism $G \rightarrow \mathrm{Aut}(X)$ is algebraic, then $\epsilon(L)$ is an extension of group schemes.

We let $\alpha \colon \mathrm{Pic}^G (X) \rightarrow \mathrm{Pic}(X)^G$ be the homomorphism that forgets the linearisation. By Cor 7.1 of \cite{Dolgachev2003}, we have that $\mathrm{ker}(\alpha) = \mathrm{Hom}_\mathrm{GpSch}(G,k^\times).$

As in Remark 3.4.3(i) in \cite{Brion}, $L \in \mathrm{Pic}(X)^G$ is linearisable if and only if the extension $\epsilon(L)$ splits as an extension of abstract groups and the resulting action of $G$ on $L$ is algebraic. Note that this latter condition is automatic when G is finite; in general, what we then require is that the extension $\epsilon(L)$ splits when viewed as an extension of algebraic groups, rather than just abstract groups. Indeed, in this situation, the induced group homomorphism $G \rightarrow \mathcal{G}(L) \rightarrow \mathrm{Aut}^{\mathbb{G}_m} (L)$ will be algebraic, and hence we obtain a linearisation of $L$. Thus, we have $\mathrm{im}(\alpha) = \mathrm{ker}(\epsilon)$, which gives our assertion. 

\end{proof}

We are now ready to prove our main theorem.

\begin{proof}[Proof of Theorem 1]
By Lemma 3, we have the following commutative diagram with exact rows:

\[
\begin{tikzcd}
0 \arrow{r} \arrow{d} & \mathrm{Hom}(G,k^\times) \arrow{r} \arrow{d} & \mathrm{Pic}^G (X) \arrow{r} \arrow{d} & \mathrm{Pic}(X)^G \arrow {r} \arrow{d} & \mathrm{Ext}_\mathrm{GpSch}^1 (G, k^\times) \arrow{d} \\
0 \arrow{r} & \mathrm{Hom}(G,k^\times) \arrow{r} & \mathrm{Pic}^G (Y) \arrow{r} & \mathrm{Pic}(Y)^G \arrow {r} & \mathrm{Ext}_\mathrm{GpSch}^1 (G, k^\times)
\end{tikzcd}
\]

The two leftmost and the rightmost vertical arrows are obviously isomorphisms. By hypothesis, the natural map $\mathrm{Pic}(X) \to \mathrm{Pic}(Y)$ is injective (resp. an isomorphism), and because it is $G$-equivariant, it induces an injection (resp. an isomorphism) of the $G$-invariant Picard groups $\mathrm{Pic}(X)^G \to \mathrm{Pic}(Y)^G$. In other words, the second vertical arrow (counting from the right) is injective (resp. an isomorphism). By the Five Lemma, the middle arrow is injective (resp. an isomorphism) as well.

\end{proof}

\begin{proof}[Proof of Corollary 2]

In view of Theorem 1, we just need to show that $\mathrm{Pic}(X) \to \mathrm{Pic}(Y)$ is injective when $\mathrm{dim}(Y) = 2$ and an isomorphism when $\mathrm{dim}(Y) \geq 3$. This follows immediately from the Grothendieck-Lefschetz Theorem for Picard groups.

\end{proof}

\newtheorem*{remark}{Remark}

\begin{remark}

We observe that if $X = \mathbb{P}_k^n$, then we can drop the hypothesis that $k$ has characteristic zero in Corollary 2, due to Cor XII.3.7 of \cite{SGA2}.

\end{remark}

\bibliographystyle{alpha}

\end{document}